\documentclass[12pt, english, a4paper]{article}
\usepackage[latin9]{inputenc}
\usepackage[T1]{fontenc}
\usepackage{graphicx}
\usepackage{amsthm}
\usepackage{amsfonts}
\usepackage{amsmath}
\usepackage{amssymb}
\usepackage{makecell}
\usepackage{comment}
\usepackage{cases}
\usepackage{mathrsfs}
\usepackage{fullpage}
\usepackage{times}
\usepackage[usenames]{color}
\usepackage[dvipsnames]{xcolor}
\usepackage{hyperref}

\newcommand{\SSS}{\mathfrak{S}}
\newcommand{\BB}{\mathfrak{B}}

\newcommand{\ol}[1]{\overline{#1}}
\newcommand{\blue}[1]{\textcolor{blue}{#1}}
\newcommand{\brown}[1]{\textcolor{brown}{#1}}
\newcommand{\red}[1]{\textcolor{red}{#1}}
\newcommand{\ogreen}[1]{\textcolor{OliveGreen}{#1}}
\newcommand{\ZZ}{\mathbb{Z}}
\newcommand{\wrp}[2]{\mathfrak{S}_{#1}^{#2}}
\newcommand{\owrp}[2]{\overline{\mathfrak{S}_{#1}^{#2}}}

\newcommand{\floor}[1]{\left\lfloor #1 \right\rfloor}

\newcommand{\GE}{ \mathsf{AExc}}

\DeclareMathOperator{\DescSet}{DescSet}
\DeclareMathOperator{\AscSet}{AscSet}
\DeclareMathOperator{\ExcS}{ExcSet}
\DeclareMathOperator{\exc}{exc}
\DeclareMathOperator{\asc}{asc}
\DeclareMathOperator{\EXC}{EXC}
\DeclareMathOperator{\des}{des}
\DeclareMathOperator{\DES}{DES}
\DeclareMathOperator{\ldes}{ldes}
\DeclareMathOperator{\lasc}{lasc}
\DeclareMathOperator{\lexc}{lexc}
\DeclareMathOperator{\bexc}{bexc}

\DeclareMathOperator{\mo}{mo}
\DeclareMathOperator{\sm}{S}
\DeclareMathOperator{\cyc}{Cyc}
\DeclareMathOperator{\rev}{REV}
\DeclareMathOperator{\Negs}{Negs}

\DeclareMathOperator{\BE}{BE}
\DeclareMathOperator{\B}{B}
\DeclareMathOperator{\Ade}{A}
\DeclareMathOperator{\AEx}{AE}
\DeclareMathOperator{\AExc}{AExc}

\theoremstyle{plain}
\newtheorem{theorem}{Theorem}
\newtheorem{corollary}[theorem]{Corollary}
\newtheorem*{corollary*}{Corollary}
\newtheorem{lemma}[theorem]{Lemma}
\newtheorem{proposition}[theorem]{Proposition}
\theoremstyle{definition}
\newtheorem{definition}[theorem]{Definition}
\newtheorem{example}[theorem]{Example}

\usepackage{authblk}
\title{A descent-excedance correspondence in colored permutation groups}
\author[1]{Hiranya Kishore Dey\thanks{\tt{hiranyadey@iisc.ac.in}}}
\author[2]{Umesh Shankar\thanks{\tt{204093001@iitb.ac.in, umeshshankar@outlook.com}}}
\author[3]{Sivaramakrishnan Sivasubramanian\thanks{\tt{krishnan@math.iitb.ac.in}}}
\affil[1]{Department of Mathematics, Indian Institute of Science, Bengaluru }
\affil[2,3]{Department of Mathematics, Indian Institute of Technology, Bombay Mumbai 400076, India}

\date{\today}
\begin{document}
\maketitle
\begin{abstract}
It is well known that descents and excedances are equidistributed in the symmetric group. We show that the descent and excedance enumerators, summed over permutations with a fixed first letter are identical when we perform a simple change of the first letter. We generalize this to type B and other colored permutation groups. We are led to defining descents and excedances through linear orders. With respect to a particular order, when the number of colors is even, we get a result that generalizes the type-B results. Lastly, we get a type B counterpart of Conger's result which refines the well known Carlitz identity.
\end{abstract}
\textbf{\small{}Keyword:}{\small{}{ Eulerian statistics, Colored permutation groups, Carlitz identity}}{\let\thefootnote\relax}

\section{Introduction}
\label{sec:intro}

For a positive integer $n$, let $[n]=\{1,2,\dots,n\}$ and let 
$\SSS_n$ denote the set of permutations of the set $[n]$.
For $\pi \in \SSS_n$ with $\pi = \pi_1, \pi_2, \ldots, \pi_n$, let 
$\DES(\pi) = \{i \in [n-1]: \pi_i > \pi_{i+1} \}$ be its set of 
descents and let $\des(\pi) = |\DES(\pi)|$ denote its number 
of descents.  For $\pi \in \SSS_n$ let 
$\EXC(\pi) = \{i \in [n]: \pi_i > i \}$ be its set of excedances 
and let $\exc(\pi) = |\EXC(\pi)|$ denote its number of excedances.
The polynomial
$A_n(t) = \sum_{\pi \in \SSS_n} t^{\des(\pi)}$ 
is defined as the $n$-th Eulerian 
polynomial and is well studied, see, for example, the book 
\cite{petersen-eulerian-nos-book} by Petersen. 
If we define $\AExc_n(t) = \sum_{\pi \in \SSS_n} t^{\exc(\pi)}$,
then, for all non negative integers $n$,
it is well known (see \cite{petersen-eulerian-nos-book})
that $A_n(t) = \AExc_n(t)$.

For a positive 
integer $n$ and for $j \in [n]$,  let $\SSS_{n,j} = 
\{ \pi \in \SSS_n: \pi_1 = j \}$  be 
the subset of $\SSS_n$ consisting of all $\pi$
that start with 
the letter $j$.  While working on riffle shuffles in decks 
with repeated cards,
Conger \cite{conger-first-letter-descent} enumerated descents over the set  $\SSS_{n,j}$. 
Let $\Ade_{n,d,j}$ denote the number of permutations in $\SSS_{n,j}$ 
that have $d$ descents.
The following alternating sum identity was shown by Conger 
\cite[Theorem 1]{conger-first-letter-descent}.  
\begin{theorem}[Conger] 
\label{thm:conger_des-alt-sum_identity}
 With the above notation, we have 
\begin{equation}
\Ade_{n,d,j}=\sum_{k\ge 0} (-1)^{d-k} \binom{n}{d-k}k^{j-1} (k+1)^{n-j}.
\end{equation}
\end{theorem}

A natural question is to ask if the number of permutations in $\SSS_n$ 
starting with the letter $j$ and having $d$ excedances  satisfies a similar
identity.  For a positive
integer $n$ and for $1 \leq j \leq n$, define $A_{n,j}(t) = 
\sum_{\pi \in \SSS_{n,j}}t^{\des(\pi)}$ 
and $\AExc_{n,j}(t) = \sum_{\pi \in \SSS_{n,j}}t^{\exc(\pi)}$ to  be the polynomials
enumerating descents 
and excedances respectively in $\SSS_{n,j}$.
For $n=6$, we get the following table (seen better on a 
colour monitor). 
\bigskip

$
\begin{array}{|l|l|l|} \hline
j & A_{6,j}(t) \mbox{ wrt descents} & \GE_{6,j}(t) \mbox{ wrt excedances} \\ \hline
1 & 1 + 26t+ 66t^2+ 26t^3+ t^4  &  1 + 26t+ 66t^2+ 26t^3+ t^4  \\ \hline \hline
2 & \blue{16t+ 66t^2+ 36t^3 + 2t^4} &  \brown{t+ 26t^2+ 66t^3+ 26t^4+ t^5} \\ \hline
3 & \red{ 8t+ 60t^2+ 48t^3+ 4t^4} &  \ogreen{ 2t+ 36t^2+ 66t^3+ 16t^4} \\ \hline
4 & 4t+ 48t^2+ 60t^3+ 8t^4  &   4t+ 48t^2+ 60t^3+ 8t^4 \\ \hline
5 & \ogreen{2t+ 36t^2+ 66t^3+ 16t^4} & \red{ 8t+ 60t^2+ 48t^3+ 4t^4} \\ \hline
6 & \brown{t+ 26t^2+ 66t^3+ 26t^4+ t^5}  & \blue{16t+ 66t^2+ 36t^3 + 2t^4}  \\ \hline
\end{array}
$

\bigskip

A corollary of our Theorem \ref{thm: criss-cross} (that we 
prove in Section \ref{sec: proofs}) 
is the following.  





\begin{corollary}
    \label{thm: des-exc-first-perm}
    When $n\ge 2$, we have $A_{n,1}(t)=\AExc_{n,1}(t)=A_{n-1}(t).$
    When $n\ge 2$ and $2\le j\le n$, we have 
$A_{n,j}(t)=\AExc_{n,n+2-j}(t).$
Thus, for a fixed $n$, the two families of polynomials $A_{n,i}(t)$ and 
$\AExc_{n,i}(t) $ are the same up to a permutation of 
the second index $i$ in the subscript. 
\end{corollary}

We generalise this result to colored permutations as follows.
For a positive integer $d$, let $[d]_0 = \{0,1,\ldots,d-1\}$ be 
a set of size $d$. 
To define descents and excedances over $\SSS_n \wr \ZZ_d$, 
we need a linear order $L$ on $[n]\times [d]_0$.  
With respect to such a linear order $L$, for a permutation $\pi$, 
we denote descents as $\ldes(\pi)$ and excedances 
as $\lexc(\pi)$ (see Section \ref{sec:prelims} for our definition).  
We denote $\SSS_n \wr \ZZ_d$ alternatively as 
$\SSS_n^d$ and call this the colored permutation group.  
For any linear order $L$,  and positive integers 
$n,d$, our general result 
is the following equidistribution of $\ldes$ and $\lexc$
on the wreath product $\SSS_n^d$.  
A similar equidistribution of descents and excedances result was 
proved by Steingrimsson \cite{steingrimsson-indexed-perms}.  Though 
Steingrimsson proves the equidistribution  of descent and excedance over $\SSS_n \wr \ZZ_d$, his definitions are different.  The following result is proved in Section \ref{sec: proofs}.

\begin{theorem}\label{thm: ldes-lexc}
For any linear order $>_L$ on $[n] \times [d]_0$, there exists a 
bijection $\Phi: \wrp{n}{d} \rightarrow \wrp{n}{d}$ such that 
$\ldes(p)=\lexc(\Phi(p))$ for all $p\in \wrp{n}{d}$. 
\end{theorem}

For some linear orders, one can get more information.   In Section 
\ref{sec:prelims}, we define
the min-one linear order on elements of colored permutations.
It will be clear when there is a single color,  for 
all positive integers $n$ that the min-one 
order reduces to the natural order on  $\mathbb{N}$
and that $\ldes$ and $\lexc$ with respect to these orders
give the usual definition of $\des$ and $\exc$ in $\SSS_n$ respectively.
Let $S_{(i,j)}$ be the set of colored permutations in 
$\wrp{n}{d}$  that start with the letter $i$ that further have 
colour $j$ for the letter $i$.
With respect to the min-one order, we show that enumerating 
$\ldes$ over $S_{(i,j)}$ gives the 
same polynomial as enumerating $\lexc$ over $S_{(\sigma(i),\mu(j))}$
for some permutations $\sigma$ and $\mu$.
Our result proved again in Section \ref{sec: proofs} is the following.

\begin{theorem}
\label{thm: criss-cross}
With respect to the min-one order, for all positive integers $n,d$, 
there exists a bijection 
$\Gamma: \wrp{n}{d} \rightarrow \wrp{n}{d}$ such that 
$\ldes(p)=\lexc(\Gamma(p))$ for all $p\in \wrp{n}{d}$ which 
additionally satisfies $\Gamma(S_{(i,j)})=S_{(n+2-i,d+1-j)}$, 
for $i\ge 2$ and $j$. When $i=1$ and $j$, we 
have $\Gamma(S_{(1,j)})=S_{(1,d+1-j)}$. 
\end{theorem}




With respect to the min-one order, for positive integers $n,d$, 
the polynomials enumerating $\ldes$ in $\SSS_n^d$ have degree $n-1$.
Recall that $\BB_n$ is the group of signed permutations over $[\pm n]$
and enumeration of $\des_B$ over $\BB_n$ gives a degree $n$ polynomial.
Thus, if we use the min-one order and have two colors,  we will 
not get a generalizations of results for $\BB_n$.

To obtain a result that generalises results known over $\BB_n$, we 
consider colored permutations with an even number of colors and 
append a zero to the front of all such permutations.  We denote 
this as $\owrp{n}{2d}$. 
The zero is also given a place in the linear order.
We thus define our second linear order called the symmetric order
with respect to which, enumeration
of $\ldes$ gives a degree $n$ polynomial for all $n,d$.  Further, this 
linear order gives the standard 
linear order on $\mathbb{Z}$ when we have two colors and further,
in this case,  $\ldes$ reduces to $\des_B$.  


We define a statistic $\bexc$ that generalises Brenti's type $B$ excedance
statistic to colored permutations with an even number of colors. 
Our second main result, proved in Section \ref{sec:coloured_perms}, is the following.

\begin{theorem}
\label{thm: ldes-bexc}
With respect to the symmetric order,  for even colored permutation groups 
$\owrp{n}{2d}$, we have a bijection $\Gamma: \owrp{n}{2d}\mapsto \owrp{n}{2d}$ 
satisfying $\ldes(p)=\bexc(\Gamma(p))$. Further, our bijection $\Gamma$ 
satisfies $\Gamma(S_{(i,j)})=S_{(i,\ol{j})}$ and $\Gamma(S_{(i,\ol{j})})=S_{(i,j)}$ 
when $1\le i\le n$ and $1\le j \le d$. 
\end{theorem}

When the number of colors is two, Theorem \ref{thm: ldes-bexc} gives us some 
corollaries.  Define the following restricted type B Eulerian numbers 
that count descents and excedances when the starting letter and its colour 
are fixed. Let
$B_{n,d,k} = \{ \pi \in \BB_n : \des_B(\pi)=d, \pi(1)=k \},$ and let  
$BE_{n,d,k}  =  \{ \pi \in \BB_n : \exc_B(\pi)=d, \pi(1)=k \}.$
We define the set $\BB_{n,k}=\{\pi \in \BB_n : \pi_1=k \}$.
 When $k\in [-n,n]\backslash \{0\}$, define the polynomials $$B_{n,k}(t)=\sum_{\pi \in \BB_{n,k}} t^{\des_B(\pi)}, \hspace{1cm} BE_{n,k}(t)=\sum_{\pi\in\BB_{n,k}}t^{\exc_B(\pi)}.$$ 
When $d=2$ we obtain the following corollary.

\begin{corollary}
    \label{thm:equidistribution_for_type_B}
	For positive integers $n,k$ with  $1 \leq k \leq n$, the following are equidistributed. 
	\begin{enumerate}
		\item
		\label{item:k=1_type_b} When $k=1$, we have $B_{n,k}(t)=BE_{n,k}(t)$ and $B_{n,-k}(t)=BE_{n,-k}(t)$. 
		\item
		\label{item:k=2_type_b} When $2 \leq k \leq n $, we have $B_{n,k}(t)=BE_{n,-k}(t)$ and $B_{n,-k}(t)=BE_{n,k}(t)$.
	\end{enumerate}
\end{corollary}

A degree $n$ polynomial $f(t) = \sum_{i=0}^n f_i t^i$ with 
$f_n \not= 0$ is said to be palindromic if $f_i = f_{n-i}$
when $0 \leq i \leq n$.  A degree $n$ palindromic polynomial 
$f(t)$ is said to be gamma positive if it can be written as
$f(t) = \sum_{i=0}^{\floor{n/2}} \gamma_it^i(1+t)^{n-2i}$
with $\gamma_i \geq 0$ for all $i$.
In \cite{barycentric-subdivision-petersen-nevo},  
Nevo, Petersen and Tenner  showed gamma positivity of the
polynomial $A_{n,j}(t) + A_{n,n+1-j}(t)$.
\begin{theorem}[Nevo, Petersen and Tenner]
\label{thm:palindromic_gamma_positive}
For a positive integer $n$ and for $1 \leq j \leq n$, 
the sum $A_{n,j}(t)+ A_{n,n+1-j}(t)$ is palindromic and 
gamma positive.
\end{theorem}

We first give recurrences for $\B_{n,d,k}$ which denotes the number 
of signed permutations in $\BB_n$ with $d$ descents that start
with the letter $k$.  Though the polynomial $B_{n,k}(t)$  is not 
palindromic on its own, in subsection \ref{subsec:gamma_pos_symm_B_Eul}, we show the 
following type B analogue of Theorem \ref{thm:palindromic_gamma_positive}.
For $1\le k\le n$, define the restricted and symmetrized 
type B Eulerian polynomials,  as 
$$\ol{B}_{n,k}(t)=B_{n,k}(t)+B_{n,\ol{k}}(t) \hspace{2 mm} \mbox{ and }
\hspace{2 mm}
\widetilde{B}_{n,k}(t)=tB_{n,k}(t)+B_{n,\ol{k}}(t).$$


\begin{theorem} \label{thm: type-b-gamma}
    For $1\le k\le n$, both $\ol B_{n,k}(t)$ and $\widetilde B_{n,k}(t)$ are gamma-positive with the centers of symmetry $n/2$ and $(n+1)/2$ respectively. 
\end{theorem}

A polynomial is said to be real-rooted if all its roots are real.
Dey \cite{dey-archiv} showed that the polynomials $A_{n,j}(t)$ are 
real-rooted when $n \geq 2$ and $1 \leq j \leq n$. Our next result 
(proved in Section \ref{sec:coloured_perms}) is a
type $B$ analogue of this result.

\begin{theorem}
\label{thm:real-rootedness-type-b}
For positive integers $n,k$ with $1 \leq k \leq n$, the polynomials 
$B_{n,k}(t)$ and $B_{n, \ol{k}}(t)$ are real-rooted.
\end{theorem}



Finally, we give a Carlitz-type identity for the polynomials $B_{n,i}(t)$. 
Our result (proved in Section \ref{sec:proof_thm-X}) is the following. This refines the Carlitz-type identity of Brenti \cite{brenti-q-eulerian-94} for the type B Coxeter group. 
\begin{theorem}
\label{thm: typeb-typea-des-carlitz}
 For positive integers $i, n$ with $i \in [n]$, we have
 \begin{equation}
 \label{eqn:carlitz-pos-i}
  \frac{B_{n,i}(t)}{ (1-t)^{n}} = \sum_{k\ge 1}\bigg( (2k+1)^{n-i}(2k)^{i-1}\bigg)t^{k}. 
\end{equation} 
When $i\in \{\ol{1},\dots, \ol{n}\}$, then
\begin{eqnarray}\label{eqn:carlitz-neg-i}
    \frac{B_{n,i}(t)}{(1-t)^{n}}
= \sum_{k\ge 1}\bigg( (2k-1)^{n-|i|}(2k)^{|i|-1}\bigg)t^{k}. 
\end{eqnarray}
\end{theorem}

\section{Preliminaries}
\label{sec:prelims}
    A colored permutation is an element of the group $\wrp{n}{d} =\ZZ_d \wr \SSS_n$. We represent a colored permutation as the product $\pi\times c$ of a permutation word $\pi=\pi_1\pi_2\dots\pi_n \in \SSS_n$ and an $n$-tuple $c=(c_1,\dots,c_n)$ with each $c_i \in \ZZ_d$.
We sometimes think of an element of $\pi\times c \in \wrp{n}{d}$, where $\pi=\pi_1\dots\pi_n \in \SSS_n$ and $c=(c_1,\dots,c_n)$ with $c_i\in \ZZ_d$, as an $n$-tuple of pairs $((\pi_1,c_1),\dots,(\pi_n,c_n))$. We will also sometimes replace the pair $(\pi_i,c_i)$ by the letter $\pi_i$ with the subscript $c_i$. This will let us write the element $\pi\times c$ as $(\pi_{1})_{c_1}\dots (\pi_{n})_{c_n}$.

Given a linear ordering $>_L$ on the elements of the set $[n] \times [d]_0$. We will define three statistics on the set of colored permutations $\wrp{n}{d}$ using this linear order.
\begin{definition}
    The \emph{$L$-descent set} of a colored permutation $p=\pi\times c$ with respect
    to the linear ordering $>_L$ is the set, $\DescSet_L(p):=\{ i\in [n-1]\ :\ (\pi_i,c_i) >_L (\pi_{i+1},c_{i+1}) \}$. The cardinality of $\DescSet_L(p)$ is denoted by $\ldes(p)$. 
\end{definition}
\begin{definition}
    The \emph{$L$-ascent set} of a colored permutation $p=\pi\times c$ with respect
    to the linear ordering $>_L$ is the set, $\AscSet_L(p):=\{ i\in [n-1]\ :\ (\pi_{i+1},c_{i+1}) >_L (\pi_{i},c_{i}) \}$.
    The cardinality of $\AscSet_L(p)$ is denoted by $\lasc(p)$.
\end{definition}

For a coloured word $w=w_1\dots w_n$, define its reverse word 
as $\rev(w)=w_n\dots w_1$. We note that both the value
along with its color  are reversed by this operation.
It is clear that $\rev:\wrp{n}{d}\mapsto \wrp{n}{d}$ is a bijection. Also, it is easy to check that $\lasc(p)=\ldes(\rev(p))$. 
\begin{definition}
    The \emph{$L$-excedance set} of a colored permutation $p=\pi \times c$ 
    with respect to the linear ordering $>_L$ is the set, $\ExcS_L(p):=\{ i\in [n]\ :\ (\pi_{\pi_i},c_{\pi_i}) >_L (\pi_{i},c_{i}) \}$.  
\end{definition}
The cardinality of $\ExcS_L(p)$ is denoted by $\lexc(p)$.

\begin{definition}
\label{defn:order_l}
    Let $L$ be the following order on $[n] \times [d]_0$
   \begin{equation*}
    1_0<_L \dots<_Ln_0<_L 1_1<_L \dots<_L n_1 <_L \cdots <_L 1_{d-1}<_L\dots<_Ln_{d-1}.
   \end{equation*}
\end{definition}

\begin{example}
\label{ex: running-example}
   Let $p=(241563, 013302)\in \SSS_6^4$.  Since $d=4$, the order $L$ is:
   \begin{equation*}
       1_0<_L \dots<_L 6_0<_L 1_1<_L \dots<_L 6_1<_L1_2<_L\dots<_L 6_2<_L 1_3<_L\dots<_L6_3.
   \end{equation*}
   Here, $5_3>_L6_0$ and therefore, $\ldes(p)=1$. Further, $4_1>_L2_0$, $5_3>_L4_1$, $3_2>_L6_0$, $1_3>_L3_2$ and so, $\lexc(p)=4$. 
\end{example}
The colored permutation $p=(\pi,c)$ written in word notation is the word $(\pi_1,c_1)\dots(\pi_n,c_n)$ or $(\pi_1)_{c_i}\dots (\pi_n)_{c_i}$ in the alphabet $[n] \times [d]_0$.
We introduce our cycle notation for colored permutations as follows. 
Let $p=(\pi,c)$ and $\pi=\cyc_1 \dots \cyc_k$ be the cycle decomposition 
of $\pi$. The cycle decomposition of $p$ is obtained by replacing the each 
element $\pi_i$ in the cycle decomposition of $\pi$ by the pair $(\pi_i,c_i)$ 
or $(\pi_i)_{c_i}$. This will be called the cycle decomposition of 
the colored permutation.

\begin{example}
    Let $p=( 3241, 0110 )\in \wrp{4}{2}$. The cycle decomposition of $p$ is $(1_0,3_1,4_0)(2_1)$ which we think is better notation than the
    equivalent $((1,0),(3,1),(4,0))((2,1))$.
\end{example}

We will look at certain ``nice'' linear orderings on $[n] \times [d]_0$ which will give us the standard linear order on $\mathbb{N}$ when we set $d=1$.
\begin{definition}
\label{defn:mo}
    The \emph{min-one order} $>_{\mo}$ on $[n] \times [d]_0$ is defined to be
    \begin{eqnarray*}
        & & 1_0 <_{\mo} 1_1 <_{\mo} \dots <_{\mo} 1_{d-1} <_{\mo} 2_0 <_{\mo} 3_0 <_{\mo} \dots <_{\mo} n_0 \\
        & & n_0 <_{\mo} 2_1 <_{\mo} \dots <_{\mo} n_1 <_{\mo} \dots <_{\mo} 2_{d-1} <_{\mo} \dots <_{\mo} n_{d-1}
    \end{eqnarray*}
\end{definition}
Note that both the orders defined in Definition \ref{defn:order_l} and 
Definition \ref{defn:mo}
specialise to the standard order on $\mathbb N$ when $d=1$. Furthermore, for these orders, when $d=1$, the statistics $\ldes(\pi)$, $\des(\pi)$, $\exc(\pi)$ and $\lexc(\pi)$ coincide.
 For $(i,j)\in [n]\times [d]_0$, let $S_{(i,j)}$ be set of all colored permutations $\pi\times c \in \wrp{n}{d}$ with $\pi_1=i$ and $c_1=j$. 

\section{Proof of Theorems \ref{thm: ldes-lexc} and \ref{thm: criss-cross}}
\label{sec: proofs}
We can prove the first main result of our work.
\begin{proof}[Proof of Theorem \ref{thm: ldes-lexc}]
 Let $p=(\pi,c) \in \wrp{n}{d}$. Let $p=\cyc_1\dots\cyc_k$ be the cycle decomposition of $p$. We give a bijective proof.
Arrange the cycles such that last element of each cycle is the largest pair $(\pi_i,c_i)$ in the linear order $L$ in that cycle. Furthermore, arrange the cycles such that the last elements are descending in the order $L$. Now, remove the parentheses  
to obtain a colored permutation $\Phi'(p)$ written in word notation. It can be checked that whenever $(\pi_{\pi_i},c_{\pi_i})>_L (\pi_i,c_i)$, an $L$-ascent is produced in $\Phi'(p)$. The map $\Phi=\rev \circ \Phi'$ is such that $\ldes(\Phi(p))=\lexc(p)$.
We need only to show that $\Phi'$ is a bijection. For a colored permutation written in word notation, mark the right to left maxima and make parentheses to obtain the colored permutation in cycle decomposition. This produces the inverse.
\end{proof}
\begin{example}
    Let $p=(241563,013302)$ with the linear order defined in Example \ref{ex: running-example}. The cycle decomposition of $p$ is $(1_32_04_15_36_03_2)$.  In Example 
    \ref{ex: running-example}, we saw that $\lexc(p)=4$. We will arrange its cycle with the largest element in the last position and get $(6_03_21_32_04_15_3)$. We remove the parenthesis and reverse the permutation. We end up with $p'=5_34_12_01_33_26_0$.
    One can check that $\ldes(p')=4$.
\end{example}
We move to the second main result of this work.
Before we see our proof, we need a definition.
Define a map $s:[n]\times [d]_0 \mapsto [n]\times [d]_0$ that takes $s(i,j)=(n+2-i,d+1-j)$ if $i\neq 1$ and when $i=1$, for all $j \in [d]_0$, define $s(1,j)=(1,d+1-j)$. It is easy to see that the map $s$ is a bijection.

\begin{proof}[Proof of Theorem \ref{thm: criss-cross}]
In cycle notation, let $p=(\pi,c) \in \wrp{n}{d}$. Let $p=\cyc_1\dots\cyc_k$ be the cycle decomposition of $p$. First, replace each pair $(i,j)$ by the pair $s(i,j)$. 
Within cycles, we order the cycles of $p$ such that last element of each cycle is the smallest pair $(\pi_i,c_i)$ in the linear order $L$ in that cycle.  Note that this is different from
the order we had in the proof of Theorem \ref{thm: ldes-lexc}.
Further, we order the cycles such that the last elements are ascending in the order $L$. Finally, we remove the parentheses. The colored permutation obtained is the required permutation.
Whenever there is an $i$ such that neither $\pi_i$ nor $\pi_{\pi_i}$ is $1$ and $$(\pi_{\pi_i},c_{\pi_i})>_{\mo}(\pi_i,c_i),$$ then we necessarily have $$(n+2-\pi_{\pi_i},n+1-c_{\pi_i})<_{\mo}(n+2-\pi_i,n+1-c_i).$$
For these indices, an $L$-excedance gets turned into an $L$-descent.
If $\pi_i=1$, then 
$$(\pi_{1},c_1)>_{\mo} (1,c_i)$$ is always true. Here, $(1,c_i)$ is the last element in the first cycle and $(\pi_{1},c_1)$ is the first element of that cycle according to our arrangement. This does not generate an $L$-descent under our bijection. 
However, we have the inequality 
$$(\pi^{-1}(1),c_{\pi^{-1}(1)})>_{\mo} (1,c_i)$$
Under our bijection, this produces an $L$-descent as $(n+2-\pi^{-1}(1),n+1-c_{\pi^{-1}(1)})>_{\mo} (1,n+1-c_i)$. Therefore, every $L$-excedance of $p$ is taken to an $L$-descent of $\Gamma(p)$, which was to be shown.
What is left to be shown is that there are no $L$-descents occurring between cycles when parentheses are being removed. This is taken care by the fact that the last element in the cycle is the always smaller in the `$\mo$' order than the first element of the succeeding cycle, by our arrangement of the cycles.
\end{proof}

\begin{example}
An illuminating example is when the number of colors is one and the order is just 
the order of $\mathbb Z$. Take a permutation written in one line notation, say $\pi=891624375$ with $\exc(\pi)=3$. Replace $i$ by $n+2-i$ for $i\neq 1$. We get $\pi'=321597846$. Write this in cycle form as specified, that is, $\pi'=(31)(2)(596784)$. Remove the parenthesis to get $\Phi(\pi)=312596784$ which has $\des(\Phi(\pi))=3$.  
\end{example}

Recall the following definitions:
 $\Ade_{n,d,k}  = \{ \pi \in \SSS_n : \des(\pi)=d, \pi(1)=k \}$   and
 	$\AEx_{n,d,k}  =  \{ \pi \in \SSS_n : \exc(\pi)=d, \pi(1)=k \}.$
\begin{proof} (Of Corollary \ref{thm: des-exc-first-perm})
When $d=1$, for all $\pi \in \SSS_n$, we clearly have $\ldes(\pi) = \des(\pi)$
and $\lexc(\pi) = \exc(\pi)$.
\end{proof}

\section{Colored permutations with an even number of colors}
\label{sec:coloured_perms}
Let $\BB_n$ be the set of permutations $\pi$ of $[\pm n]$
satisfying $\pi(-i) = -\pi(i)$.  We denote $-k$ alternatively as $\ol{k}$.
$\BB_n$ is referred to as the hyperoctahedral group or the group of 
signed permutations on $[\pm n]$. 
For $\pi \in \BB_n$ we alternatively denote $\pi(i)$ as $\pi_i$. 
For $\pi \in \BB_n$, let  $\Negs(\pi) = \{i : i > 0, 
\pi_i < 0 \}$ be the set of elements of $\pi$ which occur with a 
negative sign.  For $k \in [-n,n]$, let $\BB_{n,k}= \{ \pi \in \BB_n: \pi_1=k \}$. 
For $\pi \in \BB_n $, let $ \pi_0=0$.  As defined in Petersen's book, 
we give the following definition of type B descents:
$\des_B(\pi)= |\{ i \in [0,1,2\ldots ,n-1]: \pi_i > \pi_{i+1}\}| $ and 
let $\asc_B(\pi)= n - \des_B(\pi)$.

An immediate point to note is that the maximum number of type B descents for a signed permutation in the hyperoctahedral group is $n$ as opposed to $n-1$ for  elements of the symmetric group. This additional descent is caused by the introduction of the element $\pi_0=0$. However, in the group $\wrp{n}{2}$, the maximum number of $\ldes$ or $\lexc$ is $n-1$.

\begin{definition}
Define $I:=[-d,d]\backslash \{0\}$.  Define the set $\ol{\wrp{n}{d}}$ to be the set of colored permutations in $\wrp{n}{d}$ with the pair $(0,0)$ or $0_0$ (which we will write as just $0$), prefixed to each colored permutation.
\end{definition}

We introduce a linear order on the set $\{(0,0)\}\cup[n]\times I$.
We define a linear order that generalises the standard linear order of $\mathbb{Z}$ when we set $d=2$. 
\begin{definition}
    The symmetric order $>_{\sm}$ is the linear order on $\{(0,0)\}\cup[n]\times I$ defined as follows:
    \begin{equation*}
        n_{\ol{d}}<_{\sm} \dots<_{\sm} n_{\ol{1}}<_{\sm} (n-1)_{\ol{d}}<_{\sm}\dots<_{\sm}(n-1)_{\ol{1}}<_{\sm}\dots <_{\sm}1_{\ol{d}}<_{\sm}\dots<_{\sm}1_{\ol{1}}<_{\sm}0
    \end{equation*}
      \begin{equation*}
        0<_{\sm}1_1 <_{\sm} \dots <_{\sm} 1_d <_{\sm} 2_1 <_{\sm} \dots <_{\sm} 2_d  <_{\sm} \dots <_{\sm} n_{1} <_{\sm} \dots <_{\sm} n_{d}
    \end{equation*}
\end{definition}
The $L$-descent and $L$-ascent of a colored permutation in $\owrp{n}{d}$ are defined in the same way as for $\wrp{n}{d}$.
We define a new statistic on the set $\owrp{n}{2d}$.

\begin{definition}
    A $B$-excedance of a colored permutation $p=(\pi,c)\in \owrp{n}{2d}$, denoted by $\bexc(p)$, is the cardinality of the set $\{ i\in [n]: (\pi_{\pi_i},c_{\pi})>_{\sm} (\pi_i,c_i)\}\cup\{i\in[n]:\pi_i=i \mbox{ and } c_i < 0\}$.  
\end{definition}
Define a map from  $t: [n]\times I \mapsto [n]\times I$ that sends $(i,j)$ to $(i,\ol j)$ and $(i,\ol j)$ to $(i,j)$. With this, we are ready to prove Theorem \ref{thm: ldes-bexc}.
\begin{proof}[Proof of Theorem \ref{thm: ldes-bexc}]
  Each colored permutation $p$ in $\owrp{n}{2d}$ has the form $0,p'$ where $p'\in \wrp{n}{2d}$.  Consider the cycle decomposition of $p'$. For every pair in the colored permutation such that $\pi_i\neq i$, change the pair $(i,j)$ to the pair $t(i,j)$. Arrange the pairs in each cycle such that the last pair has the smallest first component in that cycle. Arrange the cycles such that first component of the last elements are increasing. Remove the parentheses to get a colored permutation in word form. This is a bijection as all the steps are clearly reversible.
  We claim that this is the required colored permutation by arguing that $\ldes(\Gamma(p))=\bexc(p)$.
  
  Inside each cycle, assuming $\pi_i$ is not the first component of the last element of the cycle, if $(\pi_{\pi_i},c_{\pi_i})>_{\sm} (\pi_i,c_i)$, then $t(\pi_i,c_i) >_{\sm} t(\pi_{\pi_i},c_{\pi_i})$ and this turns into an $L$-descent when we remove the parentheses. If $\pi_i$ is the first component of the last element and suppose we had $(\pi_{\pi_i},c_{\pi_i}) >_{\sm} (\pi_i,c_i)$, then since $\pi_{\pi_i}>\pi_i$ by construction, either $c_{\pi_i}$ is strictly positive. Therefore, $t(\pi_{\pi_i},c_{\pi_i})$ is negative in the second component. Since this is the first element of a block and the last element of the previous block is smaller in the first component, this will create an $L$-descent between this block and the previous block.  
  
  Similarly, for an instance where $\pi_i=i$ and $c_i<0$, by the way we have arranged the cycles, there will be an $L$-descent between blocks as the last element of the previous cycle will be smaller in the first component. Therefore, the number of $B$-excedances in $p$ is the number of $L$-descents in $\Gamma(p)$.
\end{proof}



\subsection{Real-rootedness of the type B restricted Eulerian polynomials} 
 
 We at first give recurrences satisfied by the numbers 
 $\B_{n,d,k}$ and $\BE_{n,d,k}$. 
 
 \begin{theorem}
 \label{thm:recurrence_descent_for_type_b}
 For positive integers $n,k,d$ with  $0 \leq d \leq n$ and $1 \leq k \leq n-1$, 
 \begin{eqnarray}
 \label{eqn:recurrence_descent_for_type_b}
B_{n,d,k} & = & (2d+1)B_{n-1,d,k} + (2(n-d-1)+1)B_{n-1,d-1,k}, \\
  \label{eqn:recurrence_descentnegative_for_type_b}
  B_{n,d,\ol{k}} & = & (2d-1)B_{n-1,d,\ol{k}} + (2(n-d)+1)B_{n-1,d-1,\ol{k}} .
 \end{eqnarray}
 where we have the initial conditions 
 $B_{0,0,0}=1, B_{1,0,1}=1, B_{1,1,1}=0, B_{1,1,\ol{1}}=1$ 
 and  $B_{1,0,\ol{1}}=0$. Moreover, we have $B_{n,d,n}=
 B_{n,d,\ol{n}}= 2^{n-1}A_{n,d-1}$ when 
 $1 \leq d \leq n$ and $B_{n,0,n}=B_{n,0, \ol{n}}=0$. 
 \end{theorem}
 
 \begin{proof}
 We consider \eqref{eqn:recurrence_descent_for_type_b} 
 first. 
 Let $\pi \in \BB_{n-1}$ begins with  $k$ and let $\pi$ has $d$ descents. We observe 
 the effect of inserting $n$ or $\ol{n}$ at different positions of $\pi$ other than the first position. 
 We can insert $n$ or $\ol{n}$ at a descent position or insert $n$ at the end to get an element  $\sigma \in \BB_{n}$ which begins with  $k$ and has $d$ descents. 
 
 Alternatively, one could start with an element $\pi$ of $\BB_{n-1}$ that begins with $k$ and has $d-1$ descents. 
 Here we can insert $n$ or $\ol{n}$ at an ascent position and we can insert $\ol{n}$ in the last position to get an element $\sigma$ of $\BB_n$ which begins with $k$ and has $d$ descents. 
 Thus, \eqref{eqn:recurrence_descent_for_type_b}  
 follows. 
 
 In an identical manner,
 we can prove 
 \eqref{eqn:recurrence_descentnegative_for_type_b}. Lastly, if $\pi_1 = n$ or $\pi_1= \ol{n}$, 
 then we can drop the letter $n$ and treat $\pi_2, \dots, \pi_n$ as a permutation in $\SSS_{n-1}$, yielding $B_{n,d,n}= B_{n,d,\ol{n}}= 2^{n-1}A_{n,d-1}$ for $1 \leq d \leq n$ and $B_{n,0,n}=B_{n,0, \ol{n}}=0$. 
 The proof is complete. 
 \end{proof}

 The following result follows from Theorem \ref{thm:recurrence_descent_for_type_b} and Corollary \ref{thm:equidistribution_for_type_B}. 
 
  \begin{corollary}
        	\label{thm:recurrence_excedance_for_type_b}
  	For positive integers $n,k,j$ with  $0 \leq d \leq n$ and $2 \leq k \leq n-1$, 
  	\begin{eqnarray}
  	\label{eqn:recurrence_excedance_for_type_b}
  	BE_{n,d,k} & = & (2d-1)BE_{n-1,d,k} + (2(n-d)+1)BE_{n-1,d-1,k} , \\
  	\label{eqn:recurrence_excedance_for_type_b-2}
  	BE_{n,d,-k} & = & (2d+1)BE_{n-1,d,-k} + (2(n-d-1)+1)BE_{n-1,d-1,-k}. 
  	\end{eqnarray}
    where we have the initial conditions 
    $BE_{0,0,0}=1, BE_{1,0,1}=1, BE_{1,1,1}=0, 
    BE_{1,1,-1}=1$ and  $BE_{1,0,-1}=0$.
  \end{corollary}


Let $\BB_{n,k}:=\{ \pi \in \BB_n : \pi(1)=k \}$. 
Define the restricted type B Eulerian polynomials, for $k\in [-n,n]$ to be 
$\displaystyle B_{n,k}(t)=\sum_{\pi \in \BB_{n,k}} t^{\des_B(\pi)}.$
Using Theorem \ref{thm:recurrence_descent_for_type_b}, 
we next prove 
recurrences for the polynomials $B_{n,k}(t)$.
Let $D$ be the operator $\displaystyle \frac{d}{dt}$. 

\begin{proposition}
\label{rec:type-b-rec-poly}
For positive integers $n,k$ with $1 \leq k \leq n-1$, the polynomials 
$B_{n,k}(t) $ and $B_{n,\ol{k}}(t)$ satisfy the following recurrences:
\begin{eqnarray}
\label{eqn:poly-rec-type-b}
B_{n,k}(t) & = & [1+(2n-3)t] B_{n-1, k} (t) + 2t(1-t) D B_{n-1,k} (t), \\
\label{eqn:poly-rec-type-b-neg}
B_{n, \ol{k}}(t) & = & [(2n-1)t-1] B_{n-1, \ol{k} } (t) + 2t(1-t) D B_{n-1, \ol{k} } (t)  
\end{eqnarray}
where $B_{0,0}(t)=1, B_{1,1}(t)=1, B_{1, \ol{1}}(t)=t$. 
Moreover, for $n \geq 2$, we have $B_{n,n}(t)= B_{n, \ol{n}}(t)= 2^{n-1}tA_{n-1}(t).$
\end{proposition}

\begin{proof}
We prove \eqref{eqn:poly-rec-type-b} first. Using \eqref{eqn:recurrence_descent_for_type_b}, we have 
\begin{eqnarray*}
    B_{n,k}(t) & = & \sum_{d=0}^{n} B_{n,d,k} t^d 
    =  \sum_{d=0}^n \big( (2d+1)B_{n-1,d,k} + (2(n-d-1)+1)B_{n-1,d-1,k} \big) t^d \\
    & = & \sum_{d=0}^n B_{n-1,d,k} t^d + 2t \sum_{d=0}^n dB_{n-1,d,k}t^{d-1} + (2n-3) t \sum_{d=0}^n B_{n-1,d-1,k} t^{d-1} \\
   & & - 2t^2 \sum_{d=0}^n (d-1) B_{n-1, d-1, k} t^{d-2} \\ 
    & = & [1+(2n-3)t] B_{n-1, k} (t) + 2t(1-t) D B_{n-1,k} (t)
\end{eqnarray*}
This completes the proof of \eqref{eqn:poly-rec-type-b}. We can prove \eqref{eqn:poly-rec-type-b-neg} 
in an identical way and hence we omit this. 
Finally, if $\pi_1 = n$, then we can drop the letter $n$ and treat $\pi_2, \dots, \pi_n$ as a permutation in $\SSS_{n-1}$ and this gives that 
$B_{n,n}(t)=  2^{n-1}tA_{n-1}(t).$ 
The proof is complete. 
\end{proof}

\begin{definition}
Let $f$ be a real-rooted polynomial with real roots $\alpha_1 \geq \alpha_2 \geq \dots  \geq \alpha_{\deg(f)} $ and  $g$ be a real-rooted polynomial with real roots $\beta_1 \geq \beta_2 \geq \dots \geq \beta_{\deg(g)} $. 
We say that 
$f$ interlaces $g$ if 
\[ \dots \alpha_2 \leq \beta_2 \leq \alpha_1 \leq \beta_1 .\] 
Note that in this case we must have $\deg(f) \leq \deg(g) \leq \deg(f)+1.$ If $f$ interlaces $g$ or $g$ interlaces $f$, 
then we also say that 
$f$ and $g$ have interlacing roots. 
\end{definition}

We need the following result of Obreschkoff \cite{obreschkoff}. This can also 
be found in \cite[Theorem 4.3]{hyatt-recurrences_eulerian_typeBD}. 

\begin{theorem}[Obreschkoff] 
\label{thm:real-root-interlacing}
Let $f,g \in \mathbb{R}[t]$ with $\deg(f) \leq \deg(g) \leq \deg(f)+1$. Then, $f$ interlaces $g$ 
if and only if 
$c_1f+c_2g$ has only real-roots
for all $c_1, c_2 \in \mathbb{R}$. 
\end{theorem}

We now prove Theorem \ref{thm:real-rootedness-type-b} 
using Theorem \ref{thm:real-root-interlacing}.

\begin{proof}[Proof of Theorem \ref{thm:real-rootedness-type-b}]
As $B_{n,n}(t)=B_{n, \ol{n}}(t)= 2^{n-1}tA_n(t)$, 
the real-rootedness of the polynomials $B_{n,n}(t)$ and $B_{n, \ol{n}}(t)$ follow from the real-rootedness of $A_n(t)$.  

We next prove that for $1 \leq k \leq n-1$, 
the polynomials $B_{n,k}(t)$ are real-rooted. We proceed by induction on $n$.
When, $n=1$, the assertion trivially holds. 
We assume the statement to be true for $n-1$ 
and prove that $B_{n,k}$ is real-rooted for $1 \leq k \leq n-1$.

Let $F_{n,k}(t)=tB_{n,k}(t)$. Using \eqref{eqn:poly-rec-type-b}, it can be shown that the polynomials $F_{n,k}(t)$ satisfy the following recurrence for $1 \leq k \leq n-1$:   
\begin{eqnarray*}
F_{n, k}(t) & = & [(2n-1)t-1] F_{n-1, k} (t) + 2t(1-t) D F_{n, k} (t). 
\end{eqnarray*}
By induction, the polynomial $F_{n-1,k}(t)$ is real-rooted. 
Moreover, 
the polynomial $DF_{n-1,k}(t)$ clearly interlaces $F_{n-1,k}(t)$. Therefore, $[(2n-1)t-1]F_{n-1,k}(t)$ interlaces $2t(1-t)DF_{n-1,k}(t)$. By Theorem 
\ref{thm:real-root-interlacing}, 
 the polynomial $F_{n,k}(t)$ is real-rooted. Hence, $B_{n,k}(t)$ is real-rooted. 
In an identical manner, 
one can show that the polynomials $B_{n,\ol{k}}(t)$ are real-rooted for $1 \leq k \leq n-1$. 
This completes the proof. 
\end{proof}

\subsection{Gamma positivity of symmetrized type B Eulerian polynomials}
\label{subsec:gamma_pos_symm_B_Eul}

The polynomial $B_{n,k}(t)$, in general, is not palindromic.  However, in this subsection, we  show that if  we  add two such polynomials  then we get a gamma-positive polynomial.
Our result is thus similar to the result 
\cite[Lemma 4.5]{barycentric-subdivision-petersen-nevo}
of Nevo, Petersen and Tenner.
Recall for $1 \leq k \leq n$, we had defined
$\ol{B}_{n,k}(t)=B_{n,k}(t)+B_{n,\ol{k}}(t) \hspace{2 mm} \mbox{ and }
\hspace{2 mm}
\widetilde{B}_{n,k}(t)=tB_{n,k}(t)+B_{n,\ol{k}}(t).$
We begin by showing palindromicity of both these polynomials.
\begin{proposition}
For $1\le k\le n$, the polynomials $\ol B_{n,k}(t)$ and
$\widetilde B_{n,k}(t)$ are palindromic. 
\end{proposition}

\begin{proof}
Let $k>0.$
Define $f: \BB_{n,k} \cup \BB_{n, \ol{k}} \to \BB_{n,k} \cup \BB_{n, \ol{k}} $  by 
\[ f(\pi_1,\pi_2,\dots, \pi_n) = 
 \ol{\pi_1},\ol{\pi_2},\dots,\ol{\pi_n}. \] We note that $\des(\pi)= n-\des(f(\pi))$. This proves that $\ol B_{n,k}(t)$ is palindromic. 
Let 
 $B_{n,k}(t) = a_0 + \dots +a_n t^n $. Then, 
$ B_{n,\ol{k}}(t)= a_n + \dots + a_0t^n.$
 Then, \[\widetilde B_{n,k}(t)= a_n + (a_0+a_{n-1})t+ \dots + (a_{n-1}+a_0)t^n+a_nt^{n+1}.\]
 Clearly, $\widetilde B_{n,k}(t)$ is palindromic, completing the proof. 
\end{proof}

\begin{lemma}
\label{lem:type-b-rec}
For 
positive integers $n,k$ with $1 \leq k \leq n$, we have 
\begin{eqnarray}
    B_{n+1,k}(t) & = & \displaystyle  \sum_{i=\ol{n}}^{\ol{1}} B_{n,i}(t)+ t \displaystyle \sum _{i=1}^{k-1} B_{n,i} (t) + \displaystyle \sum_{i=k}^n B_{n,i} (t), \label{eqn:first_eqn} \\
    B_{n+1,\ol{k}}(t) & = & t \displaystyle  \sum_{i=\ol{n}}^{\ol{k}} B_{n,i}(t)+  \displaystyle \sum _{i=\ol{k-1}}^{1} B_{n,i} (t) + t \displaystyle \sum_{i=1}^n B_{n,i} (t).  \label{eqn:second_eqn}
\end{eqnarray}
\end{lemma}

\begin{proof}
We consider \eqref{eqn:first_eqn} first.  Let $\pi \in \BB_{n+1,k}$ and we observe the effect of dropping $k$ from the first position of permutations in $\BB_{n+1}$. Ler $\pi'$ denote the permutation in $\BB_{\{1,2,\dots,k-1,k+1, \dots, n+1 \}}$ which we get from $\pi$ after dropping $k$. 

If $\ol{n+1} \leq \pi_2 \leq \ol{1}$, then dropping $\pi_1=k$ does not change the number of descents of $\pi$. 
If $1 \leq \pi_2 \leq k-1$,
then dropping $\pi_1=k$ decreases the number of descents of $\pi$ by $1$. Again, if $k+1 \leq  \pi_2 \leq n+1$, then dropping $\pi_1=k$ does not cause any change in the number of descents of $\pi$.
This completes the proof. As the proof of \eqref{eqn:second_eqn} follows by a similar argument, we omit it.
\end{proof}

We are now in a position to prove Theorem 
\ref{thm: type-b-gamma}.

\begin{proof} (Of Theorem \ref{thm: type-b-gamma}:)
We prove this by induction on $n$. 
It is easy to verify that $\ol B_{2,1}(t)=(1+t)^2, \widetilde B_{2,1}(t)= 2t(1+t), \ol B_{2,2}(t)=4t,$ 
and $\widetilde B_{2,2}(t)=2t(1+t).$ 
Thus the statement is true for $n=2$.
We assume the statement to be true 
for $n$ and prove this for $n+1$. 
By applying Lemma \ref{lem:type-b-rec},
we have
    \begin{eqnarray*}
        \ol B_{n+1,k}(t) &=& B_{n+1,k}(t)+B_{n+1,\ol{k}}(t)\\
        &=& (1+t) \sum_{i=\ol{n}}^{\ol{k}}B_{n,i}(t) + 2 \sum_{i=\ol{k-1}}^{\ol{1}} B_{n,i}(t)+2t \sum_{i=1}^{k-1}B_{n,i}(t) + (1+t) \sum_{i=k} ^n B_{n,i}(t)  \\
        & = & (1+t) \sum_{i=k}^n \ol B_{n,i}(t)+ 2 \sum_{i=1}^{k-1}\widetilde B_{n,i}(t).
        \label{nevo-trick}
    \end{eqnarray*}
By induction, $\ol B_{n,i}(t)$ is gamma positive with center of symmetry $n/2$ and $\widetilde B_{n,i}(t)$ is gamma positive with center of symmetry $(n+1)/2$.
Therefore, $\ol B_{n+1,k}(t)$ is gamma positive with center of symmetry $(n+1)/2$.  
In a similar way, one can show that 
\begin{equation*}
\widetilde B_{n+1}(t)
= 2t \sum_{i=k}^n \ol B_{n,i}(t) + (1+t) \sum_{i=1}^{k-1}\widetilde B_{n,i}(t). 
\end{equation*}
Thus, $\widetilde B_{n+1,k}(t)$ is gamma positive with center of symmetry $(n+2)/2$. 
Finally, we observe that $B_{n+1, n+1}(t)= B_{n+1, \ol{n+1}} (t)= \sum_{i=1}^n \widetilde B_{n,i} (t)$. 
Hence, 
both the polynomials $\ol B_{n+1, n+1}(t)$ and $\widetilde B_{n+1, n+1}(t)$ are gamma positive 
and have respective centers of symmetry $(n+1)/2$ and $(n+2)/2$. This completes the proof of the theorem. 
\end{proof}

\section{Proof of Theorem \ref{thm: typeb-typea-des-carlitz}}
\label{sec:proof_thm-X}

The goal of this section is to prove Theorem  \ref{thm: typeb-typea-des-carlitz}.
Towards this, define
$\displaystyle P_{n,i}(t)= \frac{ B_{n,i}(t)}{t^{1/2}(1-t)^n} $ when $i \geq 1$
and 
$\displaystyle Q_{n,i}(t)= \frac{ B_{n,i}(t)}{t^{3/2}(1-t)^n} $ when $i \leq -1$. 

\begin{proposition} 
\label{prop:rec-satisfied-by-restricted-typeb-anotherform} 
For positive integers $n,i$ with $1 \leq i \leq n$, the polynomials $B_{n,i}(t)$ satisfy the following recurrence relation:
\begin{equation}
\label{eqn:prop-1}
\frac{B_{n,i}(t)}{t^{1/2}(1-t)^{n}}
= \frac{d}{dt} \bigg(2t^{1/2}\frac{B_{n-1,i}(t)}{(1-t)^{n-1}}\bigg). 
\end{equation}
Therefore, we have 
$P_{n,i}(t)= \frac{d}{dt} (2tP_{n-1,i}(t)) $ with the initial condition 
$P_{i,i}(t)
=\frac{B_{i,i}(t)}{t^{1/2}(1-t)^{i}}
= \frac{2^{i-1}t^{1/2}A_{i-1}(t)}{(1-t)^{i}}$ when $i > 0$.

When  $i<0$,
the polynomials $B_{n,i}(t)$ satisfy the following recurrence relation:
\begin{equation}
\label{eqn:prop-3}
\frac{B_{n,i}(t)}{t^{3/2}(1-t)^{n}}= \frac{d}{dt} \bigg(2t^{-1/2}\frac{B_{n-1,i}(t)}{(1-t)^{n-1}}\bigg). 
\end{equation}
Therefore, we have 
$
Q_{n,i}(t)= \frac{d}{dt} (2tQ_{n-1,i}(t)) 
$
with the initial condition 
$Q_{|i|,i}(t)
=\frac{B_{|i|,i}(t)}{t^{3/2}(1-t)^{|i|}}
= \frac{2^{|i|-1}A_{|i|-1}(t)}{t^{1/2}(1-t)^{|i|}}$ when $i < 0$.
\end{proposition} 

\begin{proof}
We have  
\begin{eqnarray*}
    &&\frac{d}{dt} \left(2t^{1/2}(1-t)^{-n+1}B_{n-1,i}(t) \right)\\
    &=& t^{-1/2}(1-t)^{-n+1} B_{n-1, i}(t) + 2(n-1)t^{1/2}(1-t)^{-n}B_{n-1,i}(t) + 2t^{1/2}(1-t)^{-n+1}B'_{n-1,i}(t)\\
    &=&(t^{-1/2}(1-t)^{-n})\bigg(((1-t)+2t(n-1))B_{n-1,i}(t)+2t(1-t)B'_{n-1,i}(t)\bigg)\\
&=& (t^{-1/2}(1-t)^{-n}) \bigg( ((2n-3)t+1)B_{n-1,i}(t) + 2t(1-t)B'_{n-1,i}(t) \bigg)\\
&=& t^{-1/2}(1-t)^{-n}B_{n,i}(t)
\end{eqnarray*}
In the last line, 
we use Proposition \ref{rec:type-b-rec-poly}. 
This proves \eqref{eqn:prop-1}. We now consider the case when $i < 0$. Though proof is identical,  we give it  for the sake of completeness. In this case, we have 
\begin{eqnarray*}
&&\frac{d}{dt} \left(2t^{-1/2}(1-t)^{-n+1}B_{n-1,i}(t) \right)\\
&=& - t^{-3/2}(1-t)^{-n+1} B_{n-1, i}(t) + 2(n-1)t^{-1/2}(1-t)^{-n}B_{n-1,i}(t) + 2t^{-1/2}(1-t)^{-n+1}B'_{n-1,i}(t)\\
&=&(t^{-3/2}(1-t)^{-n})\bigg(((t-1)+2t(n-1))B_{n-1,i}(t)+2t(1-t)B'_{n-1,i}(t)\bigg)\\
&=& (t^{-3/2}(1-t)^{-n}) \bigg( ((2n-1)t-1)B_{n-1,i}(t) + 2t(1-t)B'_{n-1,i}(t) \bigg)\\
&=& t^{-3/2}(1-t)^{-n}B_{n,i}(t).
\end{eqnarray*} 
This proves \eqref{eqn:prop-3}. The proof is complete. 
\end{proof}  


\begin{proof}[Proof of Theorem \ref{thm: typeb-typea-des-carlitz}]
We first consider the case when $i>0$ and show \eqref{eqn:carlitz-pos-i} by induction on $n$. The base case is when $n=i$. 
Here, we have 
\begin{eqnarray*}
\frac{P_{i,i}(t)}{2^{i-1}}
&=& \frac{t^{1/2}A_{i-1}(t)}{(1-t)^{i}} = t^{1/2}\times \sum_{k\ge 0} (k+1)^{i-1}t^k
=  \sum_{k\ge 0} (k+1)^{i-1}t^{\frac{2k+1}{2}}.
\end{eqnarray*}

\noindent
By Proposition \ref{prop:rec-satisfied-by-restricted-typeb-anotherform},
we have 
\begin{eqnarray*}
P_{i+1,i}(t) 
&=& 2\frac{d}{dt}\left(tP_{i,i}(t)\right) 
=2^{i} \sum_{k\ge 0}   (k+1)^{i-1} \frac{2k+3}{2} t^{\frac{2k+1}{2}}. 
\mbox{  Therefore, we get} \\
P_{n,i}(t)& =& 2^{n-1}\sum_{k\ge 0}\bigg( \frac{(2k+3)^{n-i}}{2^{n-i}}(k+1)^{i-1}\bigg)t^{\frac{2k+1}{2}}. 
\mbox{ This gives} \\
\frac{B_{n,i}(t)}{ (1-t)^{n}}& =& 2^{n-1}\sum_{k\ge 0}\bigg( \frac{(2k+3)^{n-i}}{2^{n-i}} (k+1)^{i-1}\bigg)t^{k+1}. 
\end{eqnarray*} 
Or equivalently, 
\begin{eqnarray*}
    \frac{B_{n,i}(t)}{ (1-t)^{n}}& =& \sum_{k\ge 1}\bigg( (2k+1)^{n-i}(2k)^{i-1}\bigg)t^{k}. 
\end{eqnarray*}
This completes the proof of \eqref{eqn:carlitz-pos-i}. 
We now consider the case when $i < 0$ and prove \eqref{eqn:carlitz-neg-i}. Here, we have
\begin{equation*}
\frac{Q_{|i|,i}(t)}{2^{|i|-1}}
= \frac{A_{|i|-1}(t)}{t^{1/2}(1-t)^{|i|}}= 
 \frac{1}{t^{1/2}}\times \sum_{k\ge 0} (k+1)^{|i|-1}t^k. 
\end{equation*}
By Proposition \ref{prop:rec-satisfied-by-restricted-typeb-anotherform}, we have 
\begin{equation*}
Q_{|i|+1,i}(t)
= 2 \frac{d}{dt} (tQ_{|i|,i}(t))
= 2^{|i|}\sum_{k\ge 0} \bigg( \frac{(2k+1)}{2}
(k+1)^{|i|-1} \bigg) t^{\frac{2k-1}{2}}. 
\end{equation*} 
This gives us
\begin{equation*}
Q_{n,i}(t)
=  2^{n-1}\sum_{k\ge 0}\bigg( \frac{(2k+1)^{n-|i|}}{2^{n-|i|}}(k+1)^{|i|-1}\bigg)t^{\frac{2k-1}{2}}.
\end{equation*}
From this, we obtain
\begin{equation*}
\frac{B_{n,i}(t)}{(1-t)^{n}}
= 2^{n-1}\sum_{k\ge 0}\bigg( \frac{(2k+1)^{n-|i|}}{2^{n-|i|}}(k+1)^{|i|-1}\bigg)t^{k+1}. 
\end{equation*}
Or equivalently,
\begin{eqnarray*}
    \frac{B_{n,i}(t)}{(1-t)^{n}}
= \sum_{k\ge 1}\bigg( (2k-1)^{n-|i|}(2k)^{|i|-1}\bigg)t^{k}. 
\end{eqnarray*}
The proof of Theorem \ref{thm: typeb-typea-des-carlitz} is complete. 
\end{proof}

We have the following well known corollary due to Brenti \cite{brenti-q-eulerian-94}. 
\begin{corollary}[Brenti]
    For $n\ge 1$, we have 
    \begin{eqnarray}
        \frac{B_n(t)}{(1-t)^{n+1}}=\sum_{k\ge 0} (2k+1)^nt^k.
    \end{eqnarray}
\end{corollary}
\begin{proof}
Equivalently, we show that 
\begin{eqnarray*}
    \frac{B_n(t)}{(1-t)^{n}}=1+\sum_{k\ge 1}\bigg( (2k+1)^n-(2k-1)^n \bigg)t^k.
\end{eqnarray*}
    Sum over all $i$ in Theorem \ref{thm: typeb-typea-des-carlitz} and use the identity 
    \begin{eqnarray*}
        (2k+1)^n-(2k-1)^n&=&(2k+1)^n-(2k)^n+(2k)^n-(2k-1)^n\\
    &=&\sum_{i=1}^{n}\bigg( (2k+1)^{i-1}(2k)^{n-i}+(2k)^{i-1}(2k-1)^{n-i} \bigg).
    \end{eqnarray*}
    Here, we are using the identity $a^n-b^n=(a-b)(\sum_{i=0}^{n-1} a^{n-1-i}b^i)$.
    The terms on the RHS are precisely the ones that appear in Theorem \ref{thm: typeb-typea-des-carlitz}. This completes the proof.
\end{proof}

\section*{Acknowledgement}
The first author acknowledges a NBHM Post-Doctoral Fellowship (File No.
0204/10(10)/2023/R{\&}D-II/2781) during the preparation of this work. Both the first and second authors profusely thank the National Board of Higher Mathematics, India for funding.

\bibliography{main}

\begin{thebibliography}{1}

\bibitem{brenti-q-eulerian-94}
{\sc Brenti, F.}
\newblock $q$-{E}ulerian {P}olynomials {A}rising from {C}oxeter {G}roups.
\newblock {\em European Journal of Combinatorics 15\/} (1994), 417--441.

\bibitem{conger-first-letter-descent}
{\sc Conger, M.~A.}
\newblock A refinement of the {E}ulerian numbers, and the joint distribution of {$\pi(1)$} and {${\rm Des}(\pi)$} in {$S_n$}.
\newblock {\em Ars Combin. 95\/} (2010), 445--472.

\bibitem{dey-archiv}
{\sc Dey, H.~K.}
\newblock Interlacing of zeroes of certain real-rooted polynomials.
\newblock {\em Arch. Math. (Basel) 120}, 5 (2023), 457--466.

\bibitem{hyatt-recurrences_eulerian_typeBD}
{\sc Hyatt, M.}
\newblock Recurrences for {E}ulerian polynomials of type {B} and type {D}.
\newblock {\em Ann. Comb. 20}, 4 (2016), 869--881.

\bibitem{barycentric-subdivision-petersen-nevo}
{\sc Nevo, E., Petersen, T.~K., and Tenner, B.~E.}
\newblock The {{\(\gamma \)}}-vector of a barycentric subdivision.
\newblock {\em J. Comb. Theory, Ser. A 118}, 4 (2011), 1364--1380.

\bibitem{obreschkoff}
{\sc Obreschkoff, N.}
\newblock Verteilung und berechnung der nullstellen reeller polynome.
\newblock {\em VEB Deutscher Verlag der Wissenschaften, Berlin\/} (1963).

\bibitem{petersen-eulerian-nos-book}
{\sc Petersen, T.~K.}
\newblock {\em Eulerian numbers}.
\newblock Birkh\"auser Advanced Texts: Basler Lehrb\"ucher. [Birkh\"auser Advanced Texts: Basel Textbooks]. Birkh\"auser/Springer, New York, 2015.
\newblock With a foreword by Richard Stanley.

\bibitem{steingrimsson-indexed-perms}
{\sc Steingr\'imsson, E.}
\newblock Permutation statistics of indexed permutations.
\newblock {\em European J. Combin. 15}, 2 (1994), 187--205.

\end{thebibliography}
\bibliographystyle{acm}
\end{document}